\documentclass[11pt,leqno]{article}
\usepackage{geometry}
\newgeometry{vmargin={28mm}, hmargin={35mm,35mm}}   

\usepackage{xcolor}

\usepackage[hidelinks]{hyperref}

\usepackage{amsmath,amsthm,amsfonts,amssymb,latexsym,amscd,enumerate}
\usepackage{slashed}

\usepackage{palatino}
\usepackage[mathcal]{euler}

\usepackage{xy}
\xyoption{all}

\numberwithin{equation}{section}

\newtheorem{theorem}{Theorem}[section]
\newtheorem*{theorem*}{Theorem}
\newtheorem{corollary}[theorem]{Corollary}
\newtheorem{proposition}[theorem]{Proposition}
\newtheorem*{proposition*}{Proposition}

\theoremstyle{definition}
\newtheorem{definition}[theorem]{Definition}

\newtheorem{example}[theorem]{Example}

\newtheorem{remark}[theorem]{Remark}

\numberwithin{equation}{section}

\DeclareMathOperator{\Cliff}{Cliff}

\DeclareMathOperator{\supp}{supp}

\newcommand{\R}{\mathbb{R}}

\newcommand{\N}{\mathbb{N}}

\newcommand{\psidom}{\Psi^m(\mathbb{T}M, \mathbb{S})}

\newcommand{\sheafS}{\boldsymbol{{S}}}

\begin{document}

\title{Asymptotic pseudodifferential calculus and the rescaled bundle}

\author{Xiaoman Chen and Zelin Yi}

\date{}

\maketitle

\begin{abstract} 
By following a groupoid approach to pseudodifferential calculus developed by Van erp and Yuncken, we study the parallel theory on the rescaled bundle and show that the rescaled bundle gives a geometric characterization to asymptotic pseudodifferential calculus on spinor bundles by Block and Fox.
\end{abstract}


\section{Introduction}

Pseudodifferential operators have pseudo-local property, that is, their Schwartz kernels are smooth outside the diagonal. As a consequence, any pseudodifferential operator can be decomposed as a properly supported pseudodifferential operator plus a smoothing operator.
Ordinary principal symbol calculus, which is captured by the construction of tangent groupoid,  focus on the quotient space $\Psi^m(M)/\Psi^{-\infty}(M)$ where one cannot tell the difference between smoothing operator and the zero operator. However, as far as local index theory is concerned, it is precisely smoothing operators that carry  index information. More precisely, according to the Mckean-Singer formula\cite{McKeanSinger67}, the topological index can be calculated by supertrace of the smoothing operator given by the heat kernel of the Dirac operator. Therefore one need a full symbol calculus to explore the space of smoothing operators. Along this line, Widom have developed an asymptotic symbolic calculus where a smooth family of full symbols $a(x,\xi,t)$ of order $m$ is studied from the view-point of asymptotic expansion
\begin{equation}\label{eq-asym-expansion-intro}
a(x,\xi,t)\sim a_0(x,\xi)+ta_1(x,\xi)+t^2a_2(x,\xi)+\cdots.
\end{equation}
where $a_i(x,\xi)$ is a symbol of order $m-i$  and the asymptotic expansion means for all $N$, the quantity
\[
t^{-N}\left(a(x,\xi,t)-\sum_{k=0}^{N-1} t^ka_k(x,\xi)\right)
\]
converges to zero in the space of symbols of order $m-N$ as $t\to 0$.
It has the advantage that its calculus is easier than that of full symbols and, at the same time, useful aspect of smoothing operators, for example the asymptotic expansion of the operator trace at $t=0$, are preserved.

To adept this idea into the realm of local index theory, Block and Fox\cite{BlockFox90} developed the asymptotic pseudodifferential calculus on spinor bundles over compact spin manifold and use it, together with the JLO formula, to calculate Connes' cyclic cocycle.
In this paper, we recover the calculus from a geometric point-view by using a groupoid approach to pseudodifferential calculus by Van erp and Yuncken\cite{VanErpYuncken15} in the context of rescaled bundle\cite{HigsonYi19}.

Let $M$ be an even dimensional spin manifold with spinor bundle $S\to M$. The rescaled bundle $\mathbb{S}$  is a vector bundle over the tangent groupoid $\mathbb{T}M$
\begin{equation}\label{diag-bundle}
\xymatrix{
	\pi^\ast \wedge T^\ast M \ar[d] & & S^\ast \boxtimes S \ar[d] \\
	TM\times\{0\} & \sqcup & M\times M\times \mathbb{R}^\ast.
}
\end{equation}
whose restriction to $TM\times\{0\}$ is the pullback of the bundle of exterior algebras and restriction to $M\times M\times \mathbb{R}^\ast$ is the pullback of the tensor product $S^\ast \boxtimes S$. Moreover, there is an open neighborhood $\mathbb{U}$ of $TM\times \{0\}$ inside the tangent groupoid which is homeomorphic to an open neighborhood $\widetilde{\mathbb{U}}$ of $TM\times \{0\}$ inside $TM\times \mathbb{R}$. Let $\rho: TM\times \mathbb{R}\to M$ be the bundle projection, there is a natural isomorphism between the restricted rescaled bundle $\mathbb{S}|_\mathbb{U}$ and the bundle of exterior algebra $\rho^\ast \wedge T^\ast M|_{\widetilde{\mathbb{U}}}$. This fact together with the smooth structure of $S^\ast \boxtimes S$ over $M\times M\times \mathbb{R}^\ast$ completely determine the smooth structure of the rescaled bundle. 

The rescaled bundle also carries a multiplicative structure which is given by the smoothly varying maps 
\begin{equation}\label{eq-multi-map-rescaled-bundle}
\mathbb{S}_\gamma\otimes \mathbb{S}_\eta\to \mathbb{S}_{\gamma\circ \eta}
\end{equation}
where $(\gamma,\eta)$ is a composable pair of  elements in the tangent groupoid and $\gamma\circ \eta$ is their groupoid multiplication. When $\gamma=(x,X,0)$ and $\eta=(x,Y,0)$ come from $TM\times \{0\}$ part of the tangent groupoid, the multiplication map \eqref{eq-multi-map-rescaled-bundle} is explicitly computable as 
\begin{equation}\label{eq-multi-rescaled-at-t0}
\wedge T^\ast_{x} M\otimes \wedge T^\ast_{x} M \ni \alpha\otimes \beta \mapsto \exp\left(-\frac{1}{2}\kappa(X,Y)\right)\wedge \alpha\wedge\beta
\end{equation}
where $\kappa(X,Y)$ is a differential 2-form given by the symbol of curvature of the spinor bundle.

Following \cite{LescureManchonVassout17} and \cite{VanErpYuncken15}, the space of properly supported $r$-fibered distribution on the rescaled bundle $\mathcal{E}^\prime_r(\mathbb{T}M, \mathbb{S})$ is defined to be the set of continuous  $C^\infty(M\times \mathbb{R})$-module maps
$
C^\infty(\mathbb{T}M, \mathbb{S})\to C^\infty(M\times \mathbb{R})
$
and we shall consider the subspace $\psidom$ that consists of $\mathbb{P}\in \mathcal{E}^\prime_r(\mathbb{T}M, \mathbb{S})$ that satisfy the essentially homogeneous condition
\begin{equation}\label{eq-essentially-homogeneous-condition}
	\alpha_{\lambda,\ast}\mathbb{P}-\lambda^m\mathbb{P}\in C_p^\infty(\mathbb{T}M, \mathbb{S})
\end{equation}
where $\alpha_\lambda$ is a smooth action on the tangent groupoid that sends $(x,y,t)$ to $(x,y,\lambda t)$ for $t\neq 0$ and $(x,Y,0)$ to $(x,\lambda^{-1}Y, 0)$. In this paper, we shall focus on the Taylor's expansion of $\mathbb{P}\in \psidom$ at $t=0$. According to the smooth structure of the rescaled bundle, a distribution $\mathbb{P}\in \psidom$ may be restricted to an open neighborhood of $TM\times \{0\}$ and full symbol of $\mathbb{P}$ is defined by its Fourier transformation. 

Although the full symbol contains more information than the principal symbol, it is difficult to do calculation with (for example, its composition formula is complicated).
In order to preserve useful aspects of smoothing operator, instead of  the space of full symbols we shall study the space of Taylor's expansion at $t=0$
\begin{equation}\label{eq-asym-expan-distri-trinto}
	a(x,\xi,t)\sim a(x,\xi,0)+t\partial_ta(x,\xi,0)+\frac{t^2}{2}\partial^2_ta(x,\xi,0)+\cdots
\end{equation}
It has two advantages
\begin{itemize}
	\item 
	the multiplicative structure \eqref{eq-multi-map-rescaled-bundle} induces an algebra structure on $\psidom$ which in turn induces an algebra structure on the space of Taylor's expansion of full symbols. The crucial point is that the multiplication  formula of Taylor's expansion is explicitly computable and a lot easier than that of full symbols;
	\item
	If $a(x,\xi,t)$ is a full symbol of order less than or equal to $-n$, than the supertrace of the corresponding pseudodifferential kernel $\mathbb{P}$ has asymptotic expansion
	\[
	\operatorname{Str}(\mathbb{P} )\sim \sum_i t^i \cdot\left(\frac{2}{i}\right)^{n/2}(2\pi)^{-n} \int_{T^\ast M} a_i(x,\xi)d\xi dx
	\]
\end{itemize}
However, this space is restrictive in two ways:
\begin{itemize}
	\item 
	it represents only properly supported distributions;
	\item
	all $a_i(x,\xi)$'s in the expansion \eqref{eq-asym-expan-distri-trinto} are homogeneous modulo Schwartz functions.
\end{itemize}
There is an important class of pseudodifferential operators: the heat kernel that falls out of this category. We shall enlarge the space  by looking at the space of symbols $S^m$ which is defined to be the subspace of $C^\infty(T^\ast M\times \mathbb{R}, \rho^\ast \wedge T^\ast M)$ satisfy the classical  symbol estimate
\begin{equation}\label{eq-classical-symbol-estimate}
\left|\partial_x^\alpha\partial_\xi^\beta\partial_t^\gamma a(x,\xi,t)\right| \leq C\cdot (1+|\xi|+t)^{m-|\beta|},
\end{equation}
for all $\alpha, \beta,\gamma$, here $\rho: T^\ast M\times \mathbb{R}\to M$ is the bundle projection.
The space of Taylor's expansion associate with this extended symbol space retains the two advantages. Moreover, the heat kernel, which can be expressed as the fundamental solution to the differential equation
\[
\frac{\partial }{\partial \tau} + t^2D^2=0,
\]
is captured by this class of symbols. By first passing to distributions level, second to the full symbol level and third to the Taylor's expansion level of the differential equation, combine with the fact that multiplication formula of Taylor's expansion is explicitly computable, the differential equations are simplified and the solutions are explicitly computable. Under this light, the asymptotic expansion of the  heat kernel $e^{-\tau t^2D^2}$ can be calculated and the Mckean-Singer formula guarantee that the supertrace of the leading term is precisely the topological index. As a corollary, we obtain that the heat kernel $e^{-\tau t^2D^2}$ forms a smooth section of the rescaled bundle.

This paper is organized as follows: in section~\ref{sec-dis-with-coeff} we review some basic facts of the theory of distributions on manifolds with coefficient in vector bundles and their Fourier transformations; in section~\ref{sec-rescaled-bundle} we summarize the construction of the rescaled bundle over tangent groupoid for closed spin manifold; in section~\ref{sec-fibered-distirbution} we recall the theory of fibered distributions developed in \cite{LescureManchonVassout17} and, as in \cite{VanErpYuncken15}, use it to define a class of pseudodifferential operators in section~\ref{sec-psido}; in section~\ref{sec-symbols}, we study the space of Taylor's expansion of full symbols of pseudodifferential operators and show that this space has an explicitly computable multiplication formula which can be applied to heat equation and gives the asymptotic expansion of the heat kernel in section~\ref{sec-index}.

\section{Distributions with coefficient in vector bundle}\label{sec-dis-with-coeff}

Let $M$ be a closed Riemannian manifold, $E\to M$ a vector bundle over $M$ and $E^{\ast}$ the dual vector bundle. The space of distributions with coefficient in $E$ denoted by $\mathcal{D}^\prime(M, E)$ is defined to be the continuous dual of $C_c^\infty(M, E^\ast)$. Notice that we have choose a Riemannian structure on $M$ so that the density is omitted in the discussion of distributions. Let $\{U_i\}$ be an open cover of $M$ that consists of finitely many open sets and $E$ is trivial over each member, denote by $\varphi_i: E|_{U_i}\to U_i\times \mathbb{R}^n$ the trivialization and $\rho_i$ the subordinate partition of unity. For any $u_i\in \mathcal{D}^\prime(M, E)$ that has supported within some $U_i$, we have $u_i\in \mathcal{D}^\prime(U_i)\otimes \mathbb{R}^n$. In another word, $u_i$ can be written as $\sum_{I=1}^n u_{i,I}\otimes  e_{i,I}$ where $u_{i,I}$ are distributions on $U_i$ and $e_{i,I}$ are basis of $\mathbb{R}^n$.

For general $u \in \mathcal{D}^\prime(M, E)$, we have the following decomposition:
\begin{equation}\label{eq-decompose-of-distributions}
u=\sum_i \rho_i u = \sum_i  \sum_{I=1}^n \rho_i u_{i,I} \otimes e_{i,I}.
\end{equation}
Let $u\in \mathcal{D}^\prime(M)$ be a distribution on manifold $M$ and $\varphi\in C^\infty(M, E)$ be a smooth section of $E$, the product $u\cdot \varphi \in \mathcal{D}^\prime(M, E)$ is defined so that 
	\begin{equation}\label{eq-distributions-times-sections}
	\langle u\cdot \varphi, s\rangle = \langle u, \langle \varphi, s \rangle \rangle
	\end{equation}
	for all $s\in C^\infty_c(M, E^\ast)$. Here $\langle \varphi, s \rangle$ denote the compactly supported function given by $\langle \varphi, s \rangle(m) = \langle \varphi(m), s(m) \rangle_{E, E^\ast}$. Moreover, it is straightforward to check that for any $f\in C^\infty(M)$, we have $uf\cdot \varphi= u\cdot f\varphi$. Under the light of \eqref{eq-distributions-times-sections}, the equation~\eqref{eq-decompose-of-distributions} becomes $u=\sum_{i,I} \rho_iu_{i,I}\cdot e_{i,I}$. Therefore, we have  the following proposition.

\begin{proposition}\label{prop-decomposition-of-distributions}
	There is a finite set of sections $s_i$ of $E\to M$ such that any distribution $u\in \mathcal{D}^\prime(M,E)$ can be written as 
	\[
	u=\sum u_i\cdot s_i
	\]
	where $u_i\in \mathcal{D}^\prime(M)$ are some scalar distributions on $M$. \qed
\end{proposition}

Let $\pi: TM\to M$ be the tangent bundle of $M$, 
the following result is parallel to proposition~\ref{prop-decomposition-of-distributions}.

\begin{proposition}\label{prop-decomposition-distribution-tangent-bundle}
	There is a finite set of sections $s_i$ of $E\to M$ such that any distribution $u\in \mathcal{D}^\prime(TM, \pi^\ast E)$ can be written as
	\[
	u= \sum u_i \cdot \pi^\ast s_i
	\]
	where $u_i\in \mathcal{D}^\prime(TM)$ are some scalar distributions on $TM$ and $\pi^\ast s_i$ are sections of $\pi^\ast E \to TM$ given by $\pi^\ast s_i(X_m)=s_i(m)$ for all $X_m\in T_mM$. Moreover, if $u$ is a tempered distribution (distribution with compact support respectively), then $u_i$'s can be chosen to be tempered distributions (distributions with compact support respectively) on the tangent bundle.
\end{proposition}

\begin{proof}
	If $u$ is a tempered distribution with support within coordinate chart $U_i$, then $u\in \mathcal{S}^\prime(U_i\times \mathbb{R}^n) \otimes \mathbb{R}^n$. A choice of partition of unity gives the same decomposition as \eqref{eq-decompose-of-distributions} with coefficients $u_{i,I}$ scalar tempered distributions. The same argument works when $u$ is a distribution or a distribution with compact support.
\end{proof}

In this paper, we shall have occasion to consider the Fourier transformation of distributions with coefficient in vector bundles. 
 Let $u\in \mathcal{S}^\prime(TM, \pi^\ast E)$ be a tempered distribution with coefficient in $E$, its Fourier transformation, which is denote by $\widehat{u}\in \mathcal{S}^\prime(T^\ast M , \pi^\ast E)$, is defined, according to \cite[Chapter VII]{Hormanderbook1}, to satisfy
 \[
 \langle \widehat{u}, \varphi\rangle  = \langle u, \widehat{\varphi}\rangle 
 \]
 where $\varphi\in \mathcal{S}(T^\ast M, \pi^\ast E^\ast)$ is any Schwartz section and $\widehat{\varphi}\in \mathcal{S}(T M, \pi^\ast E^\ast)$ its Fourier transformation.

\begin{proposition}\label{prop-fourier-transformation-of-distribution-bundle}
	Under the light of Proposition~\ref{prop-decomposition-distribution-tangent-bundle}, the Fourier transformation of tempered distributions $\widehat{u}$ can be written as
	\[
	\widehat{u} = \sum_i \widehat{u}_i\cdot \pi^\ast s_i.
	\]
	where $\widehat{u}_i$ is the Fourier transformation of the scalar tempered distribution $u_i$. Notice here we abuse the notation $\pi^\ast s_i$ to denote both the pullback of $s_i$ to the tangent bundle and the cotangent bundle.
\end{proposition}
\begin{proof}
	Let $\varphi$ be any Schwartz section of $\pi^\ast E^\ast \to T^\ast M$, then
	\begin{align*}
	\langle u, \widehat{\varphi} \rangle &= \langle \sum u_i \cdot \pi^\ast s_i, \widehat{\varphi} \rangle\\
	&=\sum_i\langle u_i, \langle \pi^\ast s_i, \widehat{\varphi} \rangle \rangle.
	\end{align*}
According to the definition of $\pi^\ast s_i$ and the Fourier transformation, we have 
\[
\langle \pi^\ast s_i, \widehat{\varphi}\rangle(X_m) = \langle s_i(m), \widehat{\varphi}(X_m)\rangle = \mathcal{F}\left( \langle \pi^\ast s_i, \varphi \rangle\right)(X_m)
\] 
here $\mathcal{F}(\cdot)$ denote the Fourier transformation of some Schwartz functions. The above equation continues
\[
=\sum_i \langle \widehat{u}_i, \langle\pi^\ast s_i, \varphi \rangle\rangle = \langle \sum_i \widehat{u}_i \cdot \pi^\ast s_i, \varphi\rangle
\]
which completes the proof.
\end{proof}

\section{Rescaled bundle}\label{sec-rescaled-bundle}
Let $U\subset TM$ be an open neighborhood of the zero section where the exponential map is well defined and injective. Let $\widetilde{\mathbb{U}}$ be the open neighborhood of $TM\times \{0\}\subset TM\times \mathbb{R}$ given by
\[
\widetilde{\mathbb{U}} = \{(x,Y,t)\in TM\times \mathbb{R} \mid (x,-tY)\in U\}.
\]
Then the following homeomorphism determines the smooth structure of the tangent groupoid near $t=0$.
\begin{align}
	TM\times \mathbb{R} \supset \widetilde{\mathbb{U}}&\xrightarrow{\Phi} \mathbb{U}\subset \mathbb{T}M \label{eq-smooth-tangent}
	\\
	(x,Y,t)&\mapsto (x,\exp_x(-tY),t) \nonumber \\
	(x,Y,0)&\mapsto (x,Y,0),\nonumber
\end{align}
where $\mathbb{U}\subset \mathbb{T}M$ is the image of the above map.

Let $M$ be a closed spin manifold with spinor bundle $S\to M$. The rescaled bundle $\mathbb{S} \to \mathbb{T}M$ is a vector bundle over the tangent groupoid whose restriction to $M\times M\times \mathbb{R}^\ast$ is the pullback of $S^\ast \boxtimes S$ and whose restriction to $TM$ is the pullback of the bundle of exterior algebras. This information is summarized in the  diagram \eqref{diag-bundle}.
Each part in the diagram has its own smooth structure, to specify the smooth structure of the rescaled bundle it is enough to specify how does these two parts fit together.

\begin{definition} 
	Denote by 
	$A(\mathbb{T}M)\subseteq C^\infty(M\times M)[t^{-1},t]$    the $\R$-algebra of those Laurent polynomials
	\begin{equation}\label{eq-laurent-poly}
	\sum_{p\in \mathbb{Z}} f_p t^{-p}
\end{equation}
	for which  each coefficient $f_p$  is a  smooth, real-valued  function on $M\times M$ that vanishes to order $p$ on the diagonal $M\hookrightarrow M\times M$.
\end{definition}

Laurent polynomials \eqref{eq-laurent-poly} can be evaluated at any $\gamma\in \mathbb{T}M$. The evaluation is given as follows: its value at $\gamma=(x,y,\lambda)\in M\times M\times \mathbb{R}^\ast$ is given by $\sum f_p(x,y)\lambda^{-p}$ and its value at $(m,X,0)\in TM\times \{0\}$ is given by$\sum \frac{1}{p!}X^pf_p$. Moreover, the character spectrum of the associative algebra $A$ is precisely the tangent groupoid $\mathbb{T}M$, namely for any point $\gamma\in \mathbb{T}M$ there is a maximal ideal $I_\gamma\subset A(\mathbb{T}M)$ given by the collection of elements whose evaluation is zero at $\gamma$, and this viewpoint can be used to determines the manifold structure of $\mathbb{T}M$ (see \cite{HajSaeediSadeghHigson18} or \cite{HigsonYi19}).

One may take the algebra $A(\mathbb{T}M)$ as the space of "algebraic functions" on the tangent groupoid.
In the following, we shall define an $A(\mathbb{T}M)$-module denoted by $S(\mathbb{T}M)$ which can be taken as, with the previous analogy, the space of "algebraic sections" of rescaled bundle.

\begin{definition}\label{def-module}
	Denote by  $S(\mathbb{T}M)$   the complex  vector space of   Laurent polynomials
	\[
	\sum_{p \in \mathbb{Z}} \sigma _p  t^{-p}
	\]
	where each $\sigma_p$ is a smooth section of $S^\ast \boxtimes S$ of scaling order at least $p$.   
\end{definition}

We shall refer to \cite[Definition~3.3.5]{HigsonYi19} for a precise definition and detailed discussion for the notion of scaling order. However, we shall give an important example of sections with certain scaling order.

\begin{example}\label{exam-scaling-order}
	Let $n=\dim(M)$, $V$ be an open coordinate chart of $M$ over which the tangent bundle is trivial and $e_1,e_2,\cdots, e_n$ is an orthonormal frame for $TM|_V$. For any multi-index $I=\{i_1,i_2,\cdots, i_d\}$ of length $\ell(I)=d$, denote by $e_I$ the Clifford multiplication 
	\begin{equation}
	e_{i_1}e_{i_2}\cdots e_{i_d}
	\end{equation} 
	which is a local section of the bundle of Clifford algebras $\Cliff(TM)\to M$. Taken as the restriction of $S^\ast \boxtimes S\to M\times M$ to the diagonal, we can use parallel translation in the second variable to extend the local section $e_I$ to a local section of $S^\ast \boxtimes S$. We shall use the same notation $e_I$ to denote the extension, and it has scaling order $-\ell(I)$.
	In fact, any local section $\sigma$ of $S^\ast \boxtimes S$ has the form
	\[
	\sigma = \sum_I f_Ie_I
	\]
	where $f_I$ are smooth functions on $M\times M$. If the smooth function $f_I$ vanishes to order $p+\ell(I)$ for all $I$ then $\sigma$ has scaling order $p$.
\end{example}

	It is a fact that    $S(\mathbb{T}M)$   is a module over $A(\mathbb{T}M)$ by ordinary multiplication of   Laurent polynomials. 
	The fiber of $S(\mathbb{T}M)$ over $\gamma$ is defined to be
	\begin{equation*}
		S(\mathbb{T}M)\vert _\gamma  = S(\mathbb{T}M)  \big/ I_\gamma \cdot S(\mathbb{T}M).
	\end{equation*}
There is an isomorphism between $S(\mathbb{T}M)|_{(x,y,t)}$ and $S^\ast_x\otimes S_y$ which is induced by the map
\[
\varepsilon_{(x,y,\lambda)}: \sum s_pt^{-p}\mapsto \sum s_p(x,y)\lambda^{-p}
\]
for $t\neq 0$. For $t=0$, there is a similar isomorphism between $S(\mathbb{T}M)|_{(x,Y,0)}$ and $\wedge T^\ast_x M$ which is slightly more complicated and we refer the reader to \cite[Proposition~3.4.9]{HigsonYi19}. According to these two isomorphisms, there is a  map 
$$
S(\mathbb{T}M) \to \prod_{\gamma\in \mathbb{T}M} S(\mathbb{T}M)|_\gamma
$$
which sends $\sigma\in S(\mathbb{T}M)$ to its image $\widehat{\sigma}$ under the corresponding isomorphism.

 \begin{definition}
	\label{def-sheaf-of-sections}
	Denote by $\sheafS_{\mathbb{T}M}$ the sheaf on $\mathbb{T}M$ consisting of sections  
	\[
	\mathbb{T}M \ni \gamma \longmapsto \tau(\gamma) \in S(\mathbb{T}M)|_\gamma
	\]
	that are locally of the form
	\[
	\tau(\gamma)  = \sum_{j=1}^N f_j(\gamma)  \cdot \widehat \sigma_j(\gamma)
	\]
	for some $N\in \N$,  where   $f_1,\dots, f_N$ are smooth functions on $\mathbb{T}M$  and   $\sigma_1,\dots, \sigma_N$  belong to $S(\mathbb{T}M)$.
\end{definition}

It can be shown that this is a locally free sheaf and is the space of smooth sections of the rescaled bundle $\mathbb{S}\to \mathbb{T}M$. According to the construction, for any element $\sigma$ of $S(\mathbb{T}M)$, the assignment $\widehat{\sigma}$ which sends $\gamma\in \mathbb{T}M$ to $\widehat{\sigma}(\gamma)$ is a smooth section of the rescaled bundle. In fact, thanks to Example~\ref{exam-scaling-order}, the assignments $\widehat{e_I t^{\ell(I)}}$, as $I$ ranges over all multi-index $I\subset \{1,2,\cdots,n\}$, form a local frame of the rescaled bundle. The following Proposition clarify the smooth structure of the rescaled bundle over $\mathbb{U}$.

\begin{proposition}\label{prop-trivialization-of-local-rescaled-bundle}
	The pullback of the rescaled bundle along $\Phi: \widetilde{\mathbb{U}}\to \mathbb{T}M$ is isomorphic to the pullback of the bundle of exterior algebras along $\rho: TM\times \mathbb{R}\supset \widetilde{\mathbb{U}}\to M$.
\end{proposition}

\begin{proof}
	Let $V$ be a finite dimensional vector space with an inner product and an orthonormal basis $e_1,e_2,\cdots, e_n$. Denote by $q_t: \wedge^\ast V\to \Cliff(V)$ the quantization map which is given by 
	\[
	q_t(e_{i_1}\wedge e_{i_2}\wedge \cdots \wedge e_{i_k}) = t^k e_{i_1}e_{i_2}\cdots e_{i_k}
	\]
	for all $t\neq 0$. Denote by $\sigma_t: \Cliff(V) \to \wedge^\ast V$ the symbol map which is the inverse of $q_t$.
	Let $S: \rho^\ast \wedge T^\ast M \to \mathbb{S}|_{\mathbb{U}}$ be the map given by
	\[
	\left(\rho^\ast \wedge T^\ast M\right)_{(x,Y,t)}   \ni \omega
	\mapsto
	\tau_2(x, \exp_x(-tY)) q_t \omega\in \mathbb{S}_{(x, \exp_x(-tY),t)} 
	\]
	when $t\neq 0$,  and
	\[
	\left(\rho^\ast \wedge T^\ast M\right)_{(x,Y,0)}   \ni \omega \mapsto \omega \in \mathbb{S}_{(x,Y,0)}.
	\]
	here we thought $q_t \omega$ as an element in $\mathbb{S}_{(x,x,t)}$ and $\tau_2(x,\exp_x(-tY))$ is the parallel translation of the bundle $S\boxtimes S^\ast$ in the second variable from $x$ to $\exp_x(-tY)$.  
	We can also define the inverse map $T: \mathbb{S}|_{\mathbb{U}} \to \rho^\ast \wedge T^\ast M$ which sends 
	\[
	\mathbb{S}_{(x,\exp_x(-tY),t)} \ni \omega
	\mapsto 
	\sigma_t \left( \tau_2(\exp_x(-tY),x)\omega\right)\in \left(\rho^\ast \wedge T^\ast M\right)_{(x,Y,t)} 
	\]
	for $t\neq 0$, and
	\[
	\mathbb{S}_{(x,Y,0)} \ni \omega
	\mapsto 
	\omega\in \left(\rho^\ast \wedge T^\ast M\right)_{(x,Y,0)},
	\]
	here $\tau_2(\exp_x(-tY),x)$ is the parallel translation of $S\boxtimes S^\ast$ in the second variable from $\exp_x(-tY)$ to $x$.  It is easy to check that these two maps are mutually inverse.
	
	It remains to show that $T$ and $S$ are smooth. Indeed, let $V\subset M$ be an open subset over which the tangent bundle $TM$ is trivial and the trivialization is given by an orthonormal frame $e_1,e_2,\cdots, e_n$. Let $I=\{i_1,i_2,\cdots, i_{\ell(I)}\}\subset \{1,2,\cdots,n\}$ be the multi-index, denote by $\wedge^I e$ the form $e_{i_1}\wedge e_{i_2}\wedge \cdots \wedge e_{i_{\ell(I)}}$ and $e_I$ the Clifford multiplication $e_{i_1}e_{i_2}\cdots  e_{i_{\ell(I)}}$. Then the bundle $\left(\rho^\ast \wedge T^\ast M\right)|_{(TV\times \mathbb{R})\cap \widetilde{\mathbb{U}}}$ is also trivial and the constant sections $\{\wedge^I e\}_I$ form a local orthonormal frame. According to the definition, $S$ sends the constant section $\wedge^I e$ to the smooth section $\widehat{e_It^{\ell(I)}}$ of $\mathbb{S}$ which implies the smoothness of $S$. The smoothness of $T$ can be verified in the same way.
\end{proof}

The rescaled bundle also has a multiplicative structure which we shall describe now.
Let $G$ be a Lie groupoid, 
$
G^{(2)}=\{(\gamma, \eta)\in G\times G \mid s(\gamma)=r(\eta)\}
$
be the set of composable pairs. Denote by $p_1,p_2: G^{(2)}\to G$ the two coordinate projections and $m: G^{(2)}\to G$ the multiplication map.

\begin{definition}
	Let $E\to G$ be a vector bundle. We shall say that $E$ has a multiplicative structure if there is a bundle map (which we shall also denote by $m$)
	\begin{equation}\label{eq-def-of-multiplicative-bundle}
		\xymatrix{
			p_1^\ast E\otimes p_2^\ast E \ar^{\quad m}[r] \ar[d] & E \ar[d]\\
			G^{(2)} \ar^{m}[r]& G
		}
	\end{equation}
	that covers the multiplication map of the groupoid $G^{(2)}\to G$. The bundle $p_1^\ast E\otimes p_2^\ast E$ is sometimes denoted by $E\boxtimes E$.
\end{definition}

The  multiplicative structure of the rescaled bundle $m: p_1^\ast\mathbb{S}\otimes p_2^\ast \mathbb{S} \to \mathbb{S}$ is given as follows:
\begin{itemize}
	\item the restriction of $m$ away from $t=0$ is the multiplication $\mathbb{S}_{(x,y,t)}\otimes \mathbb{S}_{(y,z,t)} \to \mathbb{S}_{(x,z,t)}$ given by contracting the middle two variables;
	\item the restriction of $m$ to $t=0$ is the multiplication $\mathbb{S}_{(x,Y,0)}\otimes \mathbb{S}_{(x,Z,0)} \to \mathbb{S}_{(x,Y+Z,0)}$ given by sending $\alpha\otimes\beta \in \wedge T^\ast_xM \otimes \wedge T^\ast_xM$ to $\alpha\wedge \beta \wedge \exp(-\frac{1}{2}\kappa(Y,Z))$ where $\kappa(Y,Z)$ is the differential two form given by the symbol of curvature of the spinor bundle.
\end{itemize}
(See \cite[Section~4]{HigsonYi19})

\section{Fibered distributions}\label{sec-fibered-distirbution}

Most of the statements in this section are valid for general vector bundle over Lie groupoid that has a multiplicative structure. At the end of this section, we shall specialize to the case when $G=\mathbb{T}M$ is the tangent groupoid and $E=\mathbb{S}$ is the rescaled bundle.

\begin{definition}\label{def-r-fibered-distributions}
	Let $s: M\to B$ be a submersion, a $s$-fibered distribution on $M$ with coefficient in a vector bundle $E\to M$ is a continuous $C^\infty(B)$-module map 
	\[
	u: C^\infty_c(M, E^\ast) \to C^\infty(B)
	\]
	where the $C^\infty(B)$-module structure on $C^\infty_c(M,E^\ast)$ is given by the submersion $s$. The set of all $s$-fibered distribution is denoted by $\mathcal{D}^\prime_s(M,E)$. We can also define $s$-fibered compactly supported distribution on $M$ with coefficient in $E$ which  is given by a continuous $C^\infty(B)$-module map
	\begin{equation}\label{eq-def-fibered-dis}
	u: C^\infty(M,E^\ast) \to C^\infty(B).
	\end{equation}
	The set of all such $s$-fibered distributions will be denoted by $\mathcal{E}^\prime_s(M, E)$. If $M$ is a vector bundle over $B$ and $s$ is the bundle projection, one can talk about the $s$-fibered tempered distribution with coefficient in $E\to M$ given by a continuous $C^\infty(B)$-module map
	\[
	u: \mathcal{S}(M, E^\ast) \to C^\infty(B).
	\]
	The set of all $s$-fibered tempered distributions will be denoted by $\mathcal{S}_s^\prime(M,E)$.

\end{definition}

In the following discussion, we shall concern mainly with the case when $M$ is a Lie groupoid $G$,  $B$ is the unit space $G^{(0)}$ and $s$ is the range map $r: G\to G^{(0)}$.
We shall also have occasion to consider the case when $M=G^{(2)}$ with vector bundle $E\boxtimes E\to G^{(2)}$, $B=G^{(0)}$ and $s=r\circ m$ the composition of the multiplication map with the range map.

\begin{remark}
This definition of fibered compactly supported distribution is used in \cite{VanErpYuncken15} which is slightly different from the one used in \cite{LescureManchonVassout17}  where the authors defined the compactly supported fibered distribution by the continuous $C^\infty(B)$-module maps
\begin{equation}\label{eq-def-vanerp}
u: C^\infty(M)\to C_c^\infty(B).
\end{equation}
The continuity of $u: C^\infty(M) \to C^\infty_c(B)$ implies that there is a constant $C$, a compact subset $K\subset G$ and a finite set of seminorms $p_i(f)=\sup_{\gamma\in K} \left|\partial^\alpha f (\gamma)\right|$ on $C^\infty(M)$ such that 
\[
\sup_{x\in B}\left| u(f)(x)\right| \leq C\cdot \sum p_i(f)
\]
for all $f\in C^\infty(M)$. This implies that $u$ is supported within $K$. The definition we use here, on the other hand, does not require $u$ to have overall compact support.
\end{remark}

\begin{remark}\label{rmk-support}
If $u\in \mathcal{D}^\prime_s(M,E)$ (or $\mathcal{E}^\prime_s(M,E)$ respectively), let $f\in C^\infty_c(M,E^\ast)$ (or $C^\infty(M,E^\ast)$ respectively), we shall write 
\[
\langle u, f\rangle(x) = \langle u^x, f|_{s^{-1}(x)}\rangle
\]
where $x\in B$ and $u^x$ is the induced distribution on the fiber $s^{-1}(x)$. So, we may take the $s$-fibered distribution $u$ as a family of distributions. It is easy to see that $u^x$ are all compactly supported when $u\in \mathcal{E}^\prime_s(M,E)$.

\end{remark}

Now, let $E\to G$ be a vector bundle with multiplicative structure, we shall describe the algebra structure on the set of distributions $\mathcal{E}^\prime_r(G,E)$.
The multiplicative structure \eqref{eq-def-of-multiplicative-bundle} of $E$ induces a map $m^\ast: C^\infty(G,E^\ast) \to C^\infty(G^{(2)}, E^\ast\boxtimes E^\ast)$ which is given by
$
m^\ast f(\gamma_1,\gamma_2) = m_{\gamma_1,\gamma_2}^\ast f(\gamma_1\gamma_2)
$,
here $f$ is any smooth section of $E^\ast\to G$ and $m_{\gamma_1,\gamma_2}^\ast: E^\ast_{\gamma_1\gamma_2}\to E^\ast_{\gamma_1}\otimes E^\ast_{\gamma_2}$ is the dual of the multiplication map $m: E_{\gamma_1}\otimes E_{\gamma_2}\to E_{\gamma_1\gamma_2}$. 
At the level of fibered distributions, we can define $m_\ast: \mathcal{E}^\prime_{r\circ m}(G^{(2)}, E\boxtimes E)\to \mathcal{E}^\prime_r(G,E)$ by
$
m_\ast u(f) =u(m^\ast f),
$
for any $f\in C^\infty(G,E^\ast)$ and any $u\in \mathcal{E}^\prime_{r\circ m}(G^{(2)}, E\boxtimes E)$.
Next, consider the commutative diagram 
\[
\xymatrix{
G^{(2)} \ar_{p_2}[d] \ar^{p_1}[r]& G \ar_{s}[d] \\
G \ar^{r}[r]& G^{(0)}.
}
\]
Notice that the fibers of $p_1: G^{(2)}\to G$ correspond to the range fibers of $G$.
For any $u\in \mathcal{E}^\prime_r(G,E)$, its pullback $p_1^\ast u$ by the first projection $p_1$ is defined to be the continuous $C^\infty(G)$-module map $C^\infty(G^{(2)}, E^\ast\boxtimes E^\ast)\to C^\infty(G,E^\ast)$ which is given by 
\[
\langle p_1^\ast u,  g \rangle(\gamma)= \langle u^{s(\gamma)}, g|_{(\gamma,G^{s(\gamma)})} \rangle
\]
here the $C^\infty(G)$-module structure on $C^\infty(G^{(2)}, E^\ast\boxtimes E^\ast)$ is given by $p_1$, we take $(\gamma,G^{s(\gamma)})$ as a submanifold of $G^{(2)}$ and $u^{s(\gamma)}: C^\infty(G^{s(\gamma)},E^\ast)\to \mathbb{C}$ is the distribution induced by $u$. 

\begin{remark}\label{rmk-pullback-compact}
In fact, if $u\in \mathcal{D}^\prime_r(G,E)$, the same construction $p_1^\ast u$ is a $C^\infty(G)$-module map that sends $C^\infty_c(G^{(2)}, E^\ast\boxtimes E^\ast)$ to $C^\infty_c(G,E^\ast)$.	
\end{remark}

Let $u, v\in \mathcal{E}^\prime_r(G,E)$, then the convolution multiplication $u\ast v$ is defined to be $m_\ast (u\circ \operatorname{pr_2}^\ast v)\in \mathcal{E}^\prime_r(G,E)$. Explicitly the convolution is given by
\begin{equation}\label{eq-convolution-of-dis}
\langle u\ast v, \varphi \rangle(x) = \left\langle u^x(\gamma), \left\langle v^{s(\gamma)}(\eta), m^\ast \varphi(\gamma,\eta) \right\rangle \right\rangle
\end{equation}
for $\varphi \in C^\infty(G,E^\ast)$.
In fact, slightly more is true.

\begin{corollary}\label{coro-conv-dis-d-e}
	If $u,v\in \mathcal{D}^\prime_r(G,E)$ and at least one of them belong to $\mathcal{E}^\prime_r(G,E)$ and is properly supported, then $u\ast v$ given by the formula \eqref{eq-convolution-of-dis} is still well-defined and belongs to $\mathcal{D}^\prime_r(G,E)$.
\end{corollary}

\begin{proof}
	If $v\in \mathcal{E}_r^\prime(G,E)$ and is properly supported, $u\in  \mathcal{D}_r^\prime(G,E)$, according to the Remark~\ref{rmk-pullback-compact}, the composition $u\circ p_1^\ast v$ belongs to $\mathcal{D}^\prime_{r\circ m}(G^{(2)}, E\boxtimes E)$. As the multiplication map $m:G^{(2)}\to G$ is not necessarily proper, there is no guarantee that the map $m_\ast: \mathcal{E}^\prime_{r\circ m} (G^{(2)}, E \boxtimes E)\to \mathcal{E}^\prime_r(G,E)$ can be extended to $  \mathcal{D}^\prime_{r\circ m} (G^{(2)}, E \boxtimes E)\to \mathcal{D}^\prime_r(G,E)$. To proceed, we need to use the property of $v$ in a more serious way.
	
	According to the assumption on the support of $v$, there is a properly support function $\varphi$ with $\varphi\cdot v=v$ and $u\circ p_1^\ast v= p_2^\ast \varphi \cdot (u\circ p_1^\ast v)$. Consider the map $C^\infty(G, E^\ast)\to C^\infty(G^{(2)}, E^\ast\boxtimes E^\ast)$ which sends $f\in C^\infty(G,E^\ast)$ to $p_2^\ast \varphi\cdot m^\ast f$. If $f$ is compactly supported, then the support of $p_2^\ast \varphi\cdot m^\ast f$ is $p_2^{-1}(\supp{\varphi})\cap m^{-1}(\supp{f})$. Notice that $s\circ m = s\circ p_2$, and $\varphi$ is properly supported, it follows that $p^\ast_2 \varphi \cdot m^\ast f$ is also compactly supported. 
	
	Thus the above map takes the subspace $C^\infty_c(G,E^\ast)$ into the subspace $C_c^\infty(G^{(2)}, E^\ast\boxtimes E^\ast)$. Under this light, we have the equation
	\[
	m_\ast \left(p_2^\ast \varphi\cdot (u\circ p_1^\ast v)\right) f = \left(u\circ p_1^\ast v \right) (p_2^\ast \varphi \cdot m^\ast f)
	\]
	which proves $m_\ast(u\circ p_1^\ast v)\in \mathcal{D}^\prime_r(G,E)$.
	
	The other case, when $u\in \mathcal{E}_r^\prime(G,E)$ and is properly supported, $v\in  \mathcal{D}_r^\prime(G,E)$, can be proved in a similar way, except this time, there is a properly supported function with $u=\varphi\cdot u$ and $u\circ p^\ast_1v = p_1^\ast \varphi\cdot (u\circ p_1^\ast v)$. The fact that $r\circ m=r\circ p_1$ implies that the map $C^\infty(G,E^\ast)\to C^\infty(G^{(2)}, E^\ast \boxtimes E^\ast)$ which sends $f\in C^\infty(G,E^\ast)$ to $p_1^\ast \varphi \cdot m^\ast f\in C^\infty(G^{(2)}, E^\ast \boxtimes E^\ast)$ take the subspace of compactly supported sections to compactly supported sections. This completes the proof.
\end{proof}

Now set $G$ to be the tangent groupoid and $E$ to be the rescaled bundle.
\begin{proposition}\label{prop-multiplication-formula-at-zero-fourier}
	Let $u,v\in \mathcal{E}^\prime_r(G,E)$, write $u_0$ and $v_0$ for the restrictions $u|_{t=0}$ and $v|_{t=0}$, let $\widehat{u}_0, \widehat{v}_0$ be their Fourier transformations, then the Fourier transformation of $u_0\ast v_0$ admits the following expansion
	\begin{equation}\label{eq-multiplication-formula-at-zero-fourier}
	 \exp\left(-\frac{1}{2}\kappa(\partial_{\xi},\partial_{\eta})\right) \widehat{u}_0(\xi)\wedge\widehat{v}_0(\eta)|_{\xi=\eta}
	\end{equation}
	which coincide with the formula in \cite[Theorem~2.7]{Getzler83}. In the following discussion, we shall write $\widehat{u}_0\#_0 \widehat{v}_0$ for  \eqref{eq-multiplication-formula-at-zero-fourier}.
\end{proposition}

\begin{proof}
	Let $\pi: TM\to M$ be the bundle projection, since $u_0, v_0 \in \mathcal{E}_\pi^\prime(TM, \pi^\ast \wedge T^\ast M)$ , according to Proposition~\ref{prop-decomposition-distribution-tangent-bundle}, we may assume that $u_0$ and $v_0$ have the following forms
	\[
	u_0 = \sum_i u_i\cdot \pi^\ast s_i 
	\]
	and 
	\[
	v_0 = \sum_j v_j \cdot \pi^\ast s_j
	\]
	where $s_i$ are some sections of the bundle of exterior algebras over $M$. Then for any Schwartz section $\varphi$ of $\pi^\ast \wedge T^\ast M \to TM$ we have
	\begin{align*}
		\langle u_0\ast v_0, \widehat{\varphi} \rangle(x) &= \left\langle u_0^x(X), \left\langle  v_0^x(Y), m^\ast_{X,Y} \widehat{\varphi}(X+Y)\right\rangle \right\rangle  \\
		&=\sum_{i,j} \left\langle u_i^x(X) \cdot \pi^\ast s_i, \left\langle v^x(Y)\cdot \pi^\ast s_j, m_{X,Y}^\ast \widehat{\varphi}(X+Y) \right\rangle\right\rangle \\ 
		&= \sum_{i,j} \langle u^x_i(X)\otimes v^x_j(Y), \langle \pi^\ast s_i\otimes \pi^\ast s_j, m^\ast_{X,Y} \widehat{\varphi}(X+Y) \rangle \rangle\\
		&= \sum_{i,j} \left\langle \exp\left(-\frac{1}{2}\kappa(X,Y)\right) u^x_i(X)\otimes v^x_j(Y) \pi^\ast s_i\wedge  \pi^\ast s_j, \widehat{\varphi}(X+Y) \right\rangle .
	\end{align*}
where from the third line to the fourth line, we use the definition of $m_{X,Y}^\ast$.  The function $\widehat{\varphi}(X+Y)$ when taking as a function with two variables $\widehat{F}(X,Y)$, its inverse Fourier transformation on both variables is given by 
\begin{align*}
F(\xi,\eta) &= \int \widehat{F}(X,Y)e^{-iX\cdot \xi-Y\cdot \eta}dXdY \\
&= \int \widehat{\varphi}(X+Y) e^{-iX\cdot \xi-Y\cdot \eta}dXdY \\
&= \int \widehat{\varphi}(X) e^{-iX\cdot \xi}dX \int e^{-iY\cdot(\eta-\xi)}dY \\
&= \varphi(\xi)\cdot \delta_{\xi=\eta}
\end{align*}
According to the definition of Fourier transformation of compactly supported distributions, we conclude that  the Fourier transformation of $u_0\ast v_0$ is given by first take the Fourier transformation of 
\[
\exp(-\frac{1}{2}\kappa(X,Y))\sum_{i,j} u_i(X)\otimes v_j(Y) \cdot \pi^\ast s_i\wedge \pi^\ast s_j.
\]
on both variables and then set $\xi=\eta$. The formula is expressed precisely as in \eqref{eq-multiplication-formula-at-zero-fourier}.
\end{proof}

Analogous to Corollary~\ref{coro-conv-dis-d-e}, the same proof results in the following Corollary.
\begin{corollary}\label{coro-fourier}
	Let $u,v\in \mathcal{D}^\prime_r(G,E)$, at least one of them belong to $\mathcal{E}^\prime_r(G,E)$ with proper support and $u_0, v_0$ both belong to $\mathcal{S}^\prime(TM, \pi^\ast \wedge T^\ast M)$, then the Fourier transformation of $u_0\ast v_0$ is computed by the same formula \eqref{eq-multiplication-formula-at-zero-fourier}.
\end{corollary}

\section{Pseudodifferential operators on rescaled bundle}\label{sec-psido}



The following pseudo-local property of elements in $\psidom$ is a consequence of the properness of their support and the essentially homogeneous condition \eqref{eq-essentially-homogeneous-condition}.

\begin{proposition}\cite[Proposition~22]{VanErpYuncken15}
	Let $\mathbb{P}\in \Psi^m(\mathbb{T}M, \mathbb{S})$ for some $m\in \mathbb{R}$, then $\mathbb{P}$ is smooth on $\mathbb{T}M\backslash \mathbb{T}M^{(0)}$. \qed
\end{proposition}

As a consequence, modulo the space of smooth sections, elements in $\psidom$ may be assumed to have support within $\mathbb{U}$.

\begin{definition}\label{def-cut-off}
	A smooth function $\varphi\in C^\infty(\mathbb{T}M)$ is called a cut-off function if it is identically one on some neighborhood of $\mathbb{T}M|_{t=0}\cup  \mathbb{T}M^{(0)}$ and supported within $\mathbb{U}$.
\end{definition}

\begin{proposition}\cite[ Lemma~27]{VanErpYuncken15}
	\label{prop-localization-of-distribution}
	There is a cut-off function $\varphi\in C^\infty(\mathbb{T}M)$ such that
	for any  $\mathbb{P} \in \Psi^m(\mathbb{T}M; \mathbb{S}), \mathbb{P}^\prime=\varphi\cdot \mathbb{P} \in \Psi^m(\mathbb{T}M; \mathbb{S})$ is supported within $\mathbb{U}$ and satisfies $\mathbb{P}-\mathbb{P}^\prime\in C^\infty_p(\mathbb{T}M; \mathbb{S})$. \qed
\end{proposition}

The difference $\mathbb{P}-\mathbb{P}^\prime$ is not only smooth properly supported, but also vanishes on some neighborhood of $\mathbb{T}M|_{t=0}$. In another word, the difference  $\mathbb{P}-\mathbb{P}^\prime$ belongs to the set
\begin{equation}\label{eq-def-a}
A=\{\mathbb{P}\in C^\infty_p(\mathbb{T}M, \mathbb{S}) \mid \mathbb{P} \text{ vanishes to infinite order at } t=0\}.
\end{equation}
In the following discussion, rather than plain distributions in $\psidom$, we shall focus on the space of equivalent classes  $\Psi^m(\mathbb{T}M, \mathbb{S})/A$.

\begin{proposition}{\cite[Definition~41 and Proposition~42]{VanErpYuncken15}}
	Any equivalent class in the quotient space $\Psi^m(\mathbb{T}M, \mathbb{S})/A$ admits a representative that is homogeneous on the nose outside $[-1,1]$ and supported within $\mathbb{U}$. In another word, there is a representative $\mathbb{P}\in \psidom$ that is supported within $\mathbb{U}$ and  satisfy $\mathbb{P}_t=t^m\mathbb{P}_1$ and $\mathbb{P}_{-t}=t^m \mathbb{P}_{-1}$ for all $t>1$.
	\qed
\end{proposition}

In the following discussion, we shall always choose representatives that are homogeneous on the nose outside $[-1,1]$ and have support within $\mathbb{U}$. Combine with the Proposition~\ref{prop-trivialization-of-local-rescaled-bundle}, these choices enable us to do more explicit local calculation near $\mathbb{U}$. 

The $\alpha$-action given below the equation~\eqref{eq-essentially-homogeneous-condition} on the tangent groupoid is transformed, under the homeomorphism $\Phi$ defined in \eqref{eq-smooth-tangent}, into an action $\widetilde{\alpha}$ on $TM\times \mathbb{R}$ which is given by 
$$
\Phi^{-1}\alpha_\lambda \left(\Phi(x,Y,t)\right)=\widetilde{\alpha}_\lambda(x,Y,t).
$$ 
In another word, $\widetilde{\alpha}_\lambda$ maps $(x,Y,t)$ to $(x,\lambda^{-1}Y, \lambda t)$.

	We shall still use $\widehat{e_I t^{\ell(I)}}$ to denote the spanning sections of the pullback $\Phi^\ast \mathbb{S}\to \widetilde{\mathbb{U}}$ whose value at $(x,Y,t)$ is $e_I(x,\exp_x(-tY)) t^{\ell(I)}$ and whose value at $(x,Y,0)$ is $\wedge^I e$. It can be checked that 
	\begin{equation}\label{eq-rescaled-section}
	T\left(\widetilde{\alpha}_{\lambda,\ast} \widehat{e_I t^{\ell(I)}}(x,Y,t)\right) 
	=  \lambda^{\ell(I)} T\left(  \widehat{e_I t^{\ell(I)}}(x,Y,t)\right).
	\end{equation}
Under the light of Proposition~\ref{prop-trivialization-of-local-rescaled-bundle}, $\mathbb{P}$ is assumed to have support within $\mathbb{U}$, the pullback $\Phi^\ast\mathbb{P}$ belongs to $\mathcal{E}^\prime_r(TM\times \mathbb{R}; \Phi^\ast \mathbb{S})$ and 
can be written as
\[
\Phi^\ast \mathbb{P} = \sum_I u_I \cdot \widehat{e_It^{\ell(I)}}
\]
where $u_I$'s are $r$-fibered distributions on $TM\times \mathbb{R}$ that has support within $\widetilde{\mathbb{U}}$.  We have $T(\Phi^\ast \mathbb{P})\in \mathcal{E}^\prime_r(TM\times \mathbb{R}; \rho^\ast \wedge T^\ast M)$, thus its Fourier transformation 
\[
\widehat{\mathbb{P}}  =\widehat{T(\Phi^\ast \mathbb{P})}= \sum_I \widehat{u}_I \cdot T\left(\widehat{e_It^{\ell(I)}}\right)
\]
belongs to $C^\infty(T^\ast M\times \mathbb{R}; \rho^\ast \wedge T^\ast M)$. The essentially homogeneous condition \eqref{eq-essentially-homogeneous-condition} now becomes
\begin{equation}\label{eq-essentially-homo-condition-fourier}
\beta^\ast_\lambda \widehat{\mathbb{P}}-\lambda^m \widehat{\mathbb{P}}\in \mathcal{S}_\rho(T^\ast M\times \mathbb{R}; \rho^\ast \wedge T^\ast M)
\end{equation}
where $\beta_\lambda$ sends $(x,\xi,t)\in T^\ast M \times \mathbb{R}$ to $(x,\lambda\xi, \lambda t)$ and $\mathcal{S}_\rho$ is the space of Schwartz sections along the fiber of $\rho: T^\ast M\times \mathbb{R}\to M$.
More precisely, we have
\begin{align*}
	\beta^\ast_\lambda \widehat{\mathbb{P}}-\lambda^m \widehat{\mathbb{P}} &= \sum_I \beta_\lambda^\ast \widehat{u}_I\cdot \alpha_{\lambda,\ast} T\left(\widehat{e_I t^{\ell(I)}}\right) -\lambda^m \sum_I \widehat{u}_I \cdot T\left(\widehat{e_I t^{\ell(I)}}\right)\\
	&=\sum_I \beta_\lambda^\ast \widehat{u}_I\cdot \lambda^{\ell(I)}T\left(\widehat{e_I t^{\ell(I)}}\right)-\lambda^m \sum_I \widehat{u}_I\cdot T\left(\widehat{e_I t^{\ell(I)}}\right)\\
	&=\sum_I \lambda^{\ell(I)}\left(\beta^\ast_\lambda \widehat{u}_I-\lambda^{m-\ell(I)}\widehat{u}_I\right)\cdot  T\left(\widehat{e_I t^{\ell(I)}}\right) \in \mathcal{S}_\rho(T^\ast M\times \mathbb{R}; \rho^\ast \wedge T^\ast M)
\end{align*}
which implies that $\beta^\ast_\lambda \widehat{u}_I-\lambda^{m-\ell(I)}\widehat{u}_I$ is of Schwartz class. Thus, $u_I$ is the Schwartz kernel of some scalar pseudodifferential operator of order $m-\ell(I)$.

Let $\mathbb{P}\in \psidom$, as in \cite{VanErpYuncken15}, the operator $P=\mathbb{P}|_{t=1}: C^\infty_c(M,S)\to C^\infty_c(M,S)$ is a properly supported pseudodifferential operator on the spinor bundle $S$ and all  pseudodifferential operators on the spinor bundle can be defined in this way. The full symbol of $\mathbb{P}$ which is defined to be $\widehat{\mathbb{P}}$ can be written as
\begin{equation}\label{eq-extendable-symbol}
	\widehat{\mathbb{P}}(x,\xi,t) = \sum p_i(x,\xi,t) \otimes \omega_i
\end{equation}
where $p_i(x,\xi,t)$ is the symbol of scalar properly supported pseudodifferential operator of order $m-i$ and $\omega_i$ is a differential form of order $i$. By setting $t=1$ in \eqref{eq-extendable-symbol}, it justifies the requirement in \cite[Definition~2.2]{BlockFox90}.

\begin{remark}
	
Sometimes, for simplicity of notation, we shall write $\widetilde{\mathbb{P}}$ for  $T(\Phi^\ast \mathbb{P})$. The essentially homogeneous condition \eqref{eq-essentially-homogeneous-condition} becomes
\begin{equation}\label{eq-essentially-homogeneous-condition-pullback}
\widetilde{\alpha}_{\lambda,\ast}\widetilde{\mathbb{P}}-\lambda^m\cdot \widetilde{\mathbb{P}}\in C^\infty(TM\times \mathbb{R}; \rho^\ast \wedge T^\ast M)
\end{equation}
here again, $\rho: TM\times \mathbb{R} \to M$ denotes the bundle projection.
\end{remark}

As $\widetilde{\mathbb{P}}$ belongs to the space of fibered distributions $\mathcal{E}^\prime_\pi(TM\times \mathbb{R}, \rho^\ast \wedge T^\ast M)$ along the fiber of $\pi: TM\times \mathbb{R} \to M\times \mathbb{R}$, one may define its derivative with respect to $t$ as another fibered distribution in $\mathcal{E}^\prime_\pi(TM\times \mathbb{R}, \rho^\ast \wedge T^\ast M)$ given by 
\[
\langle \partial_t \widetilde{\mathbb{P}}, f\rangle = \partial_t \langle \widetilde{\mathbb{P}}, f\rangle- \langle \widetilde{\mathbb{P}}, \partial_t f \rangle.
\]
Under this light, one can define the Taylor's expansion of $\mathbb{P}$ to be the series
\[
\widetilde{\mathbb{P}}|_{t=0}+t\cdot  \partial_t\widetilde{\mathbb{P}}|_{t=0}+\frac{t^2}{2}\cdot \partial_t^2 \widetilde{\mathbb{P}}|_{t=0}+\cdots.
\]
In fact, the class of $\mathbb{P}$ in $\psidom/A$ is determined by the Taylor's expansion of $\widetilde{\mathbb{P}}$.

\begin{proposition}
	If $\mathbb{P}$ and $\mathbb{Q}$ are representatives of some class in $\psidom/A$, then they belong to the same class if and only if they have the same Taylor's expansion. \qed
\end{proposition}

Therefore, in the following discussion, we may take the space $\Psi^\ast(\mathbb{T}M, \mathbb{S})/A$ as the space of Taylor's expansion. Moreover, this space has a multiplicative structure: The Taylor's expansion of $\mathbb{P}\ast \mathbb{Q}$ is given by
\begin{multline*}
\widetilde{\mathbb{P}}|_{t=0}\ast \widetilde{\mathbb{Q}}|_{t=0} + t\cdot \left(\partial_t\widetilde{\mathbb{P}}|_{t=0}\ast \widetilde{\mathbb{Q}}|_{t=0}+ \widetilde{\mathbb{P}}|_{t=0}\ast \partial_t\widetilde{\mathbb{Q}}|_{t=0}\right)
\\+ 
\frac{t^2}{2}\cdot \left(\partial^2_t\widetilde{\mathbb{P}}|_{t=0}\ast \widetilde{\mathbb{Q}}|_{t=0}+ \partial_t \widetilde{\mathbb{P}}|_{t=0}\ast \partial_t \widetilde{\mathbb{Q}}|_{t=0}+ \widetilde{\mathbb{P}}|_{t=0}\ast \partial^2_t\widetilde{\mathbb{Q}}|_{t=0}\right)+\cdots.
\end{multline*}


\section{Symbols}\label{sec-symbols}

\subsection{Symbol and quantization}
	Let  $\mathbb{P}\in \Psi^m(\mathbb{T}M, \mathbb{S})$ be a representative of some class in $\psidom/A$, we shall say that the Fourier transformation of $\widetilde{\mathbb{P}}$ is a full symbol of the equivalent class $[\mathbb{P}]\in \psidom/A$. As the Fourier transformation of compactly supported distributions are smooth functions, it is easy to see that the space of full symbols are contained inside the space $C^\infty(T^\ast M\times \mathbb{R}, \rho^\ast \wedge T^\ast M)$.
	According to \cite[Proposition~43 and Corollary~45]{VanErpYuncken15}, the full symbol $a$ is equal to a genuinely homogeneous function at infinity modulo $\mathcal{S}_\rho(T^\ast M\times \mathbb{R}; \rho^\ast \wedge T^\ast M)$ and satisfies the  estimate \eqref{eq-classical-symbol-estimate}.

Of course, this definition of full symbols depend on the choice of representatives, single equivalent class may have more than one full symbols. This ambiguity can be removed by looking at the Taylor's expansion of full symbols. Indeed, let $\mathbb{P}^{\prime}$ be another representative, the difference
$
 \mathbb{P}- \mathbb{P}^{\prime}
$
has zero Taylor's expansion, and therefore its full symbols also has zero Taylor's expansion. We denote by $\operatorname{Symb}^m\subset  C^\infty(T^\ast M\times \mathbb{R};\rho^\ast \wedge T^\ast M)$ the space of all full symbols of $\mathbb{P} \in \psidom$ and by $\operatorname{Taylor-Symb}^m$ the space of all Taylor's expansion of full symbols.

\begin{definition}
	Let $\mathbb{P}\in \psidom$ be a representative, the full symbol map
	\begin{equation}\label{eq-symbol-map}
	\Sigma: \psidom/A\to \operatorname{Taylor-Symb}^m
	\end{equation}
	takes the class of $\mathbb{P}$ to the Taylor's expansion of any full symbols of $\mathbb{P}$.
\end{definition}

\begin{remark}
	There is a more direct way to calculate the full symbol from the distribution. Indeed,
	\begin{align*}
		q_t(a(x,\xi,t))s(x) &= q_t \langle \widetilde{\mathbb{P}}, e^{iY\cdot \xi} \rangle s(x) \\
		&=
		\int_{T_xM} \mathbb{P}(x,\exp_x(-tY),t) e^{iY\cdot \xi}  \tau(x,\exp_x(-tY),t) s(x) dY \\
		&=\int_M \mathbb{P}(x,y,t) e^{-i\exp_x^{-1}y\cdot \xi/t}  \tau(x,y) s(x) d\mu(y) \\
		&=\mathbb{P}_{t,y}(e^{-i\exp_x^{-1}y\cdot \xi/t} s_x(y))|_{y=x}
	\end{align*}
	where $s_x(y) = \tau(x,y)s(x)$ and $\mu$ is the Riemannian measure on $M$. By setting $t=1$, this recover the formula in \cite[Definition~2.4]{BlockFox90}.
\end{remark}


\begin{example}\label{exam-symbol-of-dsquare}
	Let $D$ be the Dirac operator on the spinor bundle $S\to M$, then Lichnerowicz formula shows that $$D^2 = -\sum_i \nabla_{e_i}^2 +s/4$$ where $s$ is the scalar curvature. Let $\mathbb{D}: C^\infty(\mathbb{T}M ,\mathbb{S})\to C^\infty(M\times \mathbb{R})$ be the $r$-fibered distribution given by the Schwartz kernel of $t^2D$ on each $r$-fiber, namely
	\[
	\left\langle\mathbb{D}, f \right\rangle(x,t) = t^2Df(x,x,t)
	\]
	where on the right hand side the Dirac operator is acting on the second variable of $f$. Similarly let $\mathbb{D}^2: C^\infty(\mathbb{T}M ,\mathbb{S})\to C^\infty(M\times \mathbb{R})$ be the $r$-fibered distribution given by
	\[
	\left\langle\mathbb{D}^2, f \right\rangle(x,t) = t^2D^2f(x,x,t)
	\]
	where the square of Dirac operator $D^2$ is again, acting on the second variable of $f$.
	The full symbol of $\mathbb{D}^2$ is given by $a(x,\xi,t) = -|\xi|^2+s/4\cdot t^2$.
\end{example}

There is also an inverse to the symbol map $
Q: \operatorname{Taylor-Symb}^m\to \psidom/A
$
which we shall describe now.
By definition, a  symbol $a(x,\xi,t)\in \operatorname{Symb}^m$ is given by the Fourier transformation of $\widetilde{\mathbb{P}}$. Applying the inverse Fourier transformation, we have
\begin{equation}\label{eq-recover-tilde-P}
	\widetilde{\mathbb{P}}(x,Y,t) = (2\pi)^{-n}\int_{T^\ast_x M} a(x,\xi,t) e^{iY\cdot \xi}  d\xi,
\end{equation}
applying the map $S$ on both sides, the equation becomes
\begin{equation}\label{eq-recover-big-p}
	\mathbb{P}(x,y,t) = (2\pi)^{-n}  \tau_2(x,y)\int_{T^\ast_x M} q_t\left( a(x,\xi,t)\right) e^{-i\xi \cdot \exp_x^{-1}(y)/t} d\xi,
\end{equation}
here $\tau_2(x,y)$ means the parallel translation in the second variable of $S\boxtimes S^\ast$ from the point $x$ to $y$. Let $\overline{u}(y)$ be the parallel translation of $u(y)$ to $S_x$. As $\mathbb{P}(x,y,1)$ is the Schwartz kernel of the properly supported pseudodifferential operator $P: C^\infty_c(M,S)\to C^\infty_c(M,S)$, we have
\begin{equation}\label{eq-recover-operator-from-symbol}
	Pu(x) = (2\pi)^{-n} \int q(a(x,\xi)) e^{-i\xi \cdot \exp_x^{-1}(y)} \overline{u}(y) d\xi dy
\end{equation}
which recover the formula of \cite[Equation~2.15]{BlockFox90}.

\begin{definition}\label{def-quantization-map}
	The quantization map 
	\[
	Q: \operatorname{Taylor-Symb}^\ast\to \Psi^\ast(\mathbb{T}M, \mathbb{S})/A
	\]
	sends the class determine by $a(x,\xi,t)$ to the class determine by the equation~\eqref{eq-recover-big-p}. By construction the quantization map is inverse to the symbol map.
\end{definition}





The space $\operatorname{Taylor-Symb}^\ast$ inherit an algebra structure from that of $\Psi^\ast(\mathbb{T}M, \mathbb{S})/A$. Its multiplication formula is particular useful and we shall now describe. Let $a,b$ be two symbols with Taylor's expansion 
$$
a\sim \sum t^{k}\cdot a_k(x,\xi)
$$ 
and 
$$
b\sim \sum t^{k}\cdot b_k(x,\xi),
$$ 
their multiplication in $\operatorname{Taylor-Symb}^m$ is given by the Taylor's expansion
\begin{equation}\label{eq-conv-formula-asymptotic}
a_0\#_0b_0+t\left(a_0\#_0b_1+a_1\#_0b_0\right)+t^2 \left(a_0\#_0b_2+a_1\#_0b_1+a_2\#_0b_2\right)+\cdots
\end{equation}
where $\#_0$ is the multiplication given in Proposition~\ref{prop-multiplication-formula-at-zero-fourier}. With this multiplication formula, the quantization map and the full symbol map are algebra homomorphisms.

If $\mathbb{P}\in \psidom$ with $m$ less than $-n$, then $\mathbb{P}$ is of trace class. Its supertrace can be calculated as
\[
\operatorname{Str}(\mathbb{P}) =(2\pi)^{-n}\left(\frac{2}{i}\right)^{n/2} \int_{T^\ast M} a(x,\xi,t)dx.
\]
Here the full symbol $a(x,\xi,t)$ take value in $\wedge^\ast T^\ast_xM$, together with $dx$, the integrand can be integrated against cotangent bundle.
Under the light of Taylor's expansion, the supertrace can be expanded as
\[
\operatorname{Str}(\mathbb{P})\sim t^k \frac{(2\pi)^{-n}}{k!}\left(\frac{2}{i}\right)^{n/2} \int_{T^\ast M} \partial_t^k a(x,\xi,0)dx.
\]

\subsection{Extended symbol space}
As mentioned in the Introduction, the symbol space $\operatorname{Symb}^m$ is restrictive in two ways, in this subsection we shall extend the symbol space as well as the quantization map and the full symbol map. 

\begin{definition}
	Let $S^m\subset C^\infty(T^\ast M\times \mathbb{R}; \rho^\ast \wedge T^\ast M)$ be the subspace whose element satisfies the symbol estimate \eqref{eq-classical-symbol-estimate}.
\end{definition}

Let $a\in S^m$, its Taylor's expansion is given by
\begin{equation}\label{eq-sym-asy}
a(x,\xi,t)\sim a(x,\xi,0)+t\partial_t a(x,\xi,0)+ \frac{t^2}{2}\partial_t^2 a(x,\xi,0)+\cdots 
\end{equation}
where the coefficient of $t^n/n!$ satisfy the relation $|\partial_x^\alpha \partial_\xi^\beta b(x,\xi)|=\mathcal{O}\left((1+|\xi|)^{m-n-|\beta|}\right)$ for all $\alpha, \beta$. Denote by $\operatorname{Taylor-S}^m$ the space of Taylor's expansion of symbols in $S^m$, the quantization map given in Definition~\ref{def-quantization-map} can be extended to  $\operatorname{Taylor-S}^m$ by using the formula \eqref{eq-recover-big-p} for each $a\in S^m$.

\begin{definition}
	Let $a\in S^m$, $\varphi$ be a cut-off function   on the tangent groupoid as in Definition~\ref{def-cut-off} and $\mathbb{Q}_\varphi(a)\in \mathcal{D}^\prime_r(\mathbb{T}M, \mathbb{S})$ be the following distribution
	\[
	\mathbb{Q}_\varphi(a) (x,y,t)= (2\pi)^{-n} \varphi(x,y,t) \tau_2(x,y)\int_{T^\ast_x M} q_t(a(x,\xi,t))e^{-i\xi\cdot \exp_x^{-1}(y)/t} d\xi.
	\]
	Let $\operatorname{Q}(S^m)$ be the set of all $\mathbb{Q}_\varphi(a)$ as $\varphi$ varies over all cut-off functions and $a$ ranges over all $S^m$. 
\end{definition}

It is easy to check that the Taylor's expansion of $\mathbb{Q}_\varphi(a)$ is given by
\begin{equation}\label{eq-distri-asy}
b_0+tb_1+\frac{t^2}{2}b_2+\cdots
\end{equation}
where $b_i\in \mathcal{D}_\pi^\prime(TM, \pi^\ast \wedge T^\ast M)$ is a fibered distribution along the bundle projection $\pi: TM\to M$ given by the oscillatory integral 
\[
b_i(x,Y) = \int \partial_t^ka(x,\xi,0) e^{iY\cdot \xi} d\xi.
\]
Let $\operatorname{Taylor-Q(S^m)}$ be the space of Taylor's expansion of elements in $Q(S^m)$.
\begin{definition}
	The extended symbol map 
	\[
	\Sigma: \operatorname{Taylor-Q(S^m)}\to \operatorname{Taylor-S^m}
	\]
	is given by sending \eqref{eq-distri-asy} to \eqref{eq-sym-asy}.
\end{definition}
It is not clear, in our context, if the space $Q(S^m)$ has an algebra structure, however, as $b_i$'s are Schwartz kernel of classical pseudodifferential operators on tangent space, the algebra structure of $\operatorname{Taylor-Q(S^m)}$ is clear. Let $\varphi, \varphi^\prime$ be two cut-off functions, the difference $\mathbb{Q}_\varphi(a)-\mathbb{Q}_{\varphi^\prime}(a)= \mathbb{Q}_{\varphi-\varphi^\prime}(a)$ has zero Taylor's expansion at $t=0$. 
\begin{definition}
The extended quantization map is a map $$Q: \operatorname{Taylor-S}^\ast \to \operatorname{Taylor-Q(S^\ast)}$$ that sends the Taylor's expansion of any $a\in S^m$ to the Taylor's expansion of $\mathbb{Q}_\varphi(a)$ for any cut-off function $\varphi$. 
\end{definition}

\begin{proposition}\label{prop-sum-of-symbol}
	Let $a_k\in C^\infty(TM, \pi^\ast \wedge T^\ast M)$ be a sequence of smooth section that satisfy
	\[
	\partial_x^\alpha \partial_\xi^\beta a_k(x,\xi) = \mathcal{O}(1+|\xi|)^{m-k-|\beta|}
	\]
	for all $\alpha, \beta$,
	then
	there exist $a(x,\xi,t)\in S^m$ which has Taylor's expansion $a(x,\xi,t)\sim \sum t^k a_k(x,\xi)$.
\end{proposition}

\begin{proof}
	(see \cite[Lemma~3.11]{BlockFox90}.) Choose $\phi\in C^\infty(\mathbb{R})$ such that $\phi(x)=0$ if $x\leq 1$ and $\phi(x)=1$ if $x\geq 2$. Let 
	$
	C_k = \max_{|\alpha|,|\beta|\leq k}\sup \left|(1+|\xi|)^{|\beta|-m}\partial_x^\alpha\partial_\xi^\beta a_k(x,\xi)\right|
	$
	and choose a decreasing sequence of positive numbers $\varepsilon_k$ that converges to zero and satisfy
	\[
	\varepsilon_k\leq\left( k!\cdot 2^i\cdot 
C_k\right)^{-2}.
	\]
	Write $\phi_k(x)$ for $\phi(\varepsilon_kx)$ and set
	\begin{equation}\label{eq-add-symbol-to-asym-expansion}
	a(x,\xi,t) = \sum_{k=0}^\infty t^k \phi_k((t|\xi|+t)^{-2})a_k(x,\xi).
	\end{equation}
	Notice that for any nonzero $t$ there is a $N\in \mathbb{N}$ such that  $\varepsilon_k(t|\xi|+t)^{-2}< 1$ for all $k>N$. Therefore, the summation \eqref{eq-add-symbol-to-asym-expansion} is locally finite and well-defined. As for the smoothness at $t=0$, we observe that the differential of \eqref{eq-add-symbol-to-asym-expansion} at any $(x,\xi,t)\in T^\ast M\times \mathbb{R}^\ast$ is given by
	\begin{equation}\label{eq-differential-symbol-estiamte}
	\partial_x^\alpha\partial_\xi^\beta\partial_t^\gamma
	a(x,\xi,t)
	=
	\sum_{k=0}^{\infty} C_{\gamma_1\gamma_2}C_{\beta_1\beta_2}\cdot k!\cdot t^{k-|\gamma_1|} \partial_\xi^{\beta_2}\partial_t^{\gamma_2} \left( \phi_k((t|\xi|+t)^{-2})\right) \partial_x^\alpha \partial_\xi^{\beta_1} a_k(x,\xi).
	\end{equation}
The sum $\sum_{k=0}^\infty$ can be split into two parts $\sum_{k\leq |\alpha|+|\beta|+|\gamma|+1}$ and $\sum_{k\geq |\alpha|+|\beta|+|\gamma|+2}$ where the first part is finite and clearly continuous at $t=0$. 

For the second part we may assume that $\varepsilon_k(t|\xi|+t)^{-2}\geq 1$, or equivalently, $t\leq \sqrt{\varepsilon_k}(|\xi|+1)^{-1}\leq \sqrt{\varepsilon_k}$, and
the factor $t^{|\gamma_2|}\partial_\xi^{\beta_2}\partial_t^{\gamma_2} \left( \phi_k((t|\xi|+t)^{-2})\right)$ is controlled by $(1+|\xi|)^{-|\beta_2|}$ times a constant independent of $t, k$ and $|\xi|$.  So, the absolute value of each term in the sum $\sum_{k\geq |\alpha|+|\beta|+|\gamma|+2}$ is controlled by a constant that only depends on $\alpha, \beta,\gamma$ times
$
 k!t^{k-|\gamma|} C_k \cdot (1+|\xi|)^{m-|\beta|-k}.
$
And the choice of $\varepsilon_k$ ensures that  each summand in the second part is controlled by $t\cdot 2^{-k}\cdot (1+|\xi|)^{m-\beta|-k}$ This proves the smoothness as well as the symbol estimate \eqref{eq-classical-symbol-estimate}.

Now, we shall verify the Taylor's expansion condition. Indeed,  since $\phi_k((t|\xi|+t)^{-2})-1$ vanishes to infinite order at $t=0$, by setting $\alpha,\beta=0$ in \eqref{eq-differential-symbol-estiamte}, we have $\partial_t^\gamma a(x,\xi,0)=a_\gamma(x,\xi)$. This completes the proof.

\end{proof}

The space $\operatorname{Taylor-S}^\ast$ can be given a multiplication formula by the formula~\eqref{eq-conv-formula-asymptotic}. In fact, this multiplication formula turn  $\operatorname{Taylor-S}^\ast$ into a graded algebra. Indeed, let $a, b\in  C^\infty(TM, \pi^\ast \wedge T^\ast M)$ satisfy $\partial_x^\alpha \partial_\xi^\beta a(x,\xi) = \mathcal{O}(1+|\xi|)^{m-|\beta|}$ and $\partial_x^\alpha \partial_\xi^\beta b(x,\xi) = \mathcal{O}(1+|\xi|)^{n-|\beta|}$ for all $\alpha, \beta$. Then 
\[
a\#_0 b(x,\xi) =\exp\left(-\frac{1}{2}\kappa(\partial_\xi,\partial_\eta)\right) a(x,\xi) b(x,\eta)|_{\xi=\eta}
\]
can be expanded to a finite sum where each term is some derivatives with respect to $\xi$ or $\eta$. Its highest order term is $a(x,\xi)b(x,\xi)$ which satisfy 
$$
\partial_x^\alpha \partial_\xi^\beta \left(a(x,\xi)b(x,\xi)\right) = \mathcal{O}(1+|\xi|)^{m+n-|\beta|}.
$$
Therefore according to the Proposition~\ref{prop-sum-of-symbol}, the multiplication of $a\in \operatorname{Taylor-S}^m$ and $b\in \operatorname{Taylor-S}^n$ belongs to $\operatorname{Taylor-S}^{m+n}$.  

\begin{proposition}\label{prop-main-tool}
	The extended full symbol map is an algebra isomorphism whose inverse is given by the extended quantization map.
\end{proposition}

\begin{proof}
	Let $\sum t^k\mathbb{P}_k$ and $\sum t^k\mathbb{Q}_k$ be two Taylor's expansions in $\operatorname{Taylor-Q(S^\ast)}$ whose extended full symbol are $\sum t^k a_k$ and $\sum t^k b_k$ respectively, their multiplication is given by
	\[
	\mathbb{P}_0\ast \mathbb{Q}_0+t\left(\mathbb{P}_1\ast \mathbb{Q}_0+ \mathbb{P}_0\ast \mathbb{Q}_1\right)+\frac{t^2}{2}\left(\mathbb{P}_2\ast \mathbb{Q}_0+2\mathbb{P}_1\ast \mathbb{Q}_1+\mathbb{P}_0\ast \mathbb{Q}_2\right)+\cdots
	\]
	whose Fourier transformation according to Corollary~\ref{coro-fourier} is given by
	\begin{equation}\label{eq-expansion-product}
		a_0\#_0 b_0+ t(a_1 \#_0 b_0+ a_0\#_0 b_1)+\frac{t^2}{2}(a_2\#_0b_0+ 2a_1\#_0 b_1+ a_0\#_0 b_2)+\cdots
	\end{equation}
which coincide with the algebra structure of $\operatorname{Taylor-S^\ast}$.
\end{proof}

\section{Index theory}\label{sec-index}

Consider the differential equation 
\begin{equation}\label{eq-heat-groupoid}
\frac{\partial f}{\partial\tau} + \mathbb{D}^2 f = 0.
\end{equation}
When restricting on each $t\neq 0$ fiber on the tangent groupoid, the fundamental solution, namely the solution with initial condition $f|_{\tau=0}=\delta$, is the heat kernel $e^{-\tau t^2D^2}(x,y)\in C^\infty(M\times M, S\boxtimes S^\ast)$. An important problem is that for fixed $\tau \neq 0$ whether this solution extend smoothly to $t=0$ as a smooth section of the rescaled bundle.

Assume that the solution $f$ belongs to $Q(S^\ast)$, and whose full symbol has Taylor's expansion, 
$$
a^\tau_0(x,\xi)+\frac{t^2}{2}a^\tau_2(x,\xi)+\frac{t^4}{24}a^\tau_4(x,\xi)+\cdots,
$$
then, according to the Proposition~\ref{prop-main-tool}, the full symbol of $\mathbb{D}^2 f$ has asymptotic expansion
\[
-|\xi|^2\#_0 a^\tau_0(x,\xi)+\frac{t^2}{2}\left(-|\xi|^2\#_0 a^\tau_2(x,\xi)+\frac{s}{2}a^\tau_0(x,\xi)\right)+\cdots
\]
where the convolution $\#_0$ is given by the formula \eqref{eq-multiplication-formula-at-zero-fourier}. As $|\xi|^2$ is a polynomial of order $2$,
the convolution $-|\xi|^2\#_0 a^\tau_{2k}$ has only three terms and can be explicitly calculated as
\begin{multline*}
-|\xi|^2\#_0 a_{2k}^\tau(x,\xi) = -|\xi|^2a_{2k}^\tau(x,\xi)+\kappa(e_i,e_j)\xi_i\partial_{\xi_j}a_{2k}^\tau(x,\xi)
\\
+\frac{1}{4}\kappa(e_i,e_p)\kappa(e_p,e_j) \partial_{\xi_i} \partial_{\xi_j} a_{2k}^\tau(x,\xi).
\end{multline*}
At the level of full symbols, the differential equation~\eqref{eq-heat-groupoid} becomes
\begin{multline*}
	\frac{\partial a_{2k}^\tau}{\partial \tau}(x,\xi)+\left(|\xi|^2-\kappa(\xi, \partial_\xi)-\frac{1}{4}\kappa\wedge \kappa(\partial_\xi, \partial_\xi)\right) a_{2k}^\tau(x,\xi)
	\\
	+ s\left(k-1\right)\left(k-\frac{1}{2}\right)a^\tau_{2k-2}(x,\xi)=0.
\end{multline*}
with initial condition $a^\tau_{2k}|_{\tau=0}=0$ for all $k\geq 1$ and
\begin{equation}\label{eq-heat-leading}
\frac{\partial a_0^\tau}{\partial \tau}(x,\xi)+\left(|\xi|^2-\kappa(\xi, \partial_\xi)-\frac{1}{4}\kappa\wedge \kappa(\partial_\xi, \partial_\xi)\right) a_0^\tau(x,\xi)=0,
\end{equation}
with initial condition $a_0^\tau|_{\tau=0}=1$.
The fundamental solution $A^\tau(\xi,\eta)$ of \eqref{eq-heat-leading} is given by the Mehler's formula which is Schwartz in $\xi$ and one can solve $a_{2k}^\tau(x,\xi)$ recursively by the formula
\[
a_{2k}^\tau(x,\xi) = -s\left(k-1\right)\left(k-\frac{1}{2}\right)\int_0^\tau A^{\tau-\tau^\prime}(\xi,\eta) a^{\tau^\prime}_{2k-2}(x,\eta)d\tau^\prime d\eta
\]
which is also Schwartz in $\xi$.

According to Proposition~\ref{prop-sum-of-symbol}, there is a full symbol $a(x,\xi,t)$ which has asymptotic expansion 
\begin{equation}\label{eq-asym-heat}
a^\tau_0(x,\xi)+\frac{t^2}{2}a^\tau_2(x,\xi)+\frac{t^4}{24}a^\tau_4(x,\xi)+\cdots.
\end{equation}
It is easy to see that the full symbol $a(x,\xi,t)$ is an approximate solution to the equation~\eqref{eq-heat-groupoid}, namely
\[
\frac{\partial a}{\partial \tau}(x,\xi,t)-\left(-|\xi|^2+\frac{s}{4}t^2\right)\ast_0 a(x,\xi,t) 
\]
vanishes to infinite order at $t=0$. By passing to quantization, we get an approximate solution $f^\prime$ to \eqref{eq-heat-groupoid}.

\begin{proposition}
	Let $$g=\frac{\partial f^\prime}{\partial \tau}+\mathbb{D}^2f^\prime$$ be the smooth section of the rescaled bundle that vanishes to infinite order at $t=0$, then the differential equation
	\[
	\frac{\partial f}{\partial\tau} + \mathbb{D}^2 f = g
	\]
	with the initial condition $f|_{\tau=0}=0$ has solution belongs to $C^\infty(\mathbb{T}M, \mathbb{S})$.
\end{proposition}

\begin{proof}
	For any nonzero $t$, the equation can be solved by the formula
	\[
	f(\tau)(x,y,t) = \int_0^\tau e^{-t^2(\tau-\tau^\prime)D^2}(x,y) g(\tau^\prime)(x,y,t) d\tau^\prime.
	\]
	Then as $g$ vanishes to infinite order at $t=0$, the section $f(\tau)\in C^\infty(M\times M\times \mathbb{R}^\ast; S^\ast \boxtimes S)$ vanishes to infinite order as $t\to 0$. The section $f(\tau)$ can be locally written as a finite sum
	\[
	f(\tau)(x,y,t) = \sum t^nf_i(x,y,t) s_i(x,y)
	\]
	where $f_i$ are smooth functions on $M\times M\times \mathbb{R}^\ast$ that vanishes as $t\to 0$, $s_i$ are sections of $S^\ast \boxtimes S\to M\times M$ and $n$ is the dimension of $M$. Then each $f_i$ can be extended to a smooth function on the tangent groupoid and $s_i(x,y)t^n$ is an element of the module $S(\mathbb{T}M)$ defined in Definition~\ref{def-module} and is a smooth section of the rescaled bundle. Overall, $f(\tau)$ can be extended to the smooth section of the rescaled bundle $\mathbb{S}$.
\end{proof}

We see that $f-f^\prime$ is a genuine solution to the heat equation \eqref{eq-heat-groupoid} and it has the same Taylor's expansion as the approximate solution $f^\prime$. 
Therefore, the image of \eqref{eq-asym-heat} under the extended quantization map is precisely the asymptotic expansion of the heat kernel.
 
\bibliography{Refs} 
\bibliographystyle{alpha}

\noindent{\small School of Mathematical Sciences, Fudan University, Shanghai, 200433, China}

\smallskip

\noindent{\small Email: xchen@fudan.edu.cn
}
\smallskip

\noindent {\small  School of Mathematical Sciences, Tongji University, Shanghai, 200092, China}

\smallskip

\noindent{\small Email:  zelin@tongji.edu.cn.}

\end{document}